\documentclass[12pt,reqno]{amsart}
\usepackage{enumerate, latexsym, amsmath, amsfonts, amssymb, amsthm, color}
\def\pmod #1{\ ({\rm{mod}}\ #1)}

\def\l{\left}
\def\r{\right}
\def\bg{\bigg}
\def\({\bg(}
\def\){\bg)}

\def\Ack{\medskip\noindent {\bf Acknowledgment}}
\theoremstyle{plain}
\newtheorem{theorem}{Theorem}

\newtheorem{lemma}{Lemma}

\newtheorem{conjecture}{Conjecture}
\theoremstyle{definition}

\theoremstyle{remark}
\newtheorem{remark}{Remark}

 \vspace{4mm}

\begin{document}
\title
[{Proof of some conjectural congruences}]
{Proof of some conjectural congruences involving Domb numbers}

\author
[Guo-Shuai Mao and Yan Liu] {Guo-Shuai Mao and Yan Liu}

\address {(Guo-Shuai Mao) Department of Mathematics, Nanjing
University of Information Science and Technology, Nanjing 210044,  People's Republic of China\\
{\tt maogsmath@163.com  } }

\address{(Yan Liu) Department of Mathematics, Nanjing
University of Information Science and Technology, Nanjing 210044,  People's Republic of China\\
{\tt 1325507759@qq.com  } }

\keywords{Congruences; Domb numbers; binary quadratic forms; $p$-adic Gamma function; Gamma function.
\newline \indent The first author is the corresponding author. This research was supported by the Natural Science Foundation of China (grant 12001288).}

 \begin{abstract} In this paper, we mainly prove the following conjectures of Z.-H. Sun \cite{SH2}:
 Let $p>3$ be a prime. If $p\equiv1\pmod3$ and $p=x^2+3y^2$, then we have
$$
\sum_{k=0}^{p-1}\frac{D_k}{4^k}\equiv\sum_{k=0}^{p-1}\frac{D_k}{16^k}\equiv4x^2-2p-\frac{p^2}{4x^2}\pmod{p^3},
$$
and if $p\equiv2\pmod3$, then
$$
\sum_{k=0}^{p-1}\frac{D_k}{4^k}\equiv-2\sum_{k=0}^{p-1}\frac{D_k}{16^k}\equiv\frac{p^2}2\binom{\frac{p-1}2}{\frac{p-5}6}^{-2} \pmod{p^3},
$$
where $D_n=\sum_{k=0}^n\binom{n}k^2\binom{2k}k\binom{2n-2k}{n-k}$ stands for the $n$th Domb number.
\end{abstract}

\maketitle

\section{Introduction}
\setcounter{lemma}{0}
\setcounter{theorem}{0}
\setcounter{corollary}{0}
\setcounter{remark}{0}
\setcounter{equation}{0}
\setcounter{conjecture}{0}

It is known that the Domb numbers which were introduced by Domb are defined by the following sequence:
$$D_n=\sum_{k=0}^n\binom{n}k^2\binom{2k}{k}\binom{2n-2k}{n-k}.$$
The $n$th Domb number also means the number of $2n$-step polygons on diamond lattice. Such sequence appears as coefficients in various series for $1/\pi$. For example, from \cite{CCL} we know that
$$
\sum_{n=0}^\infty\frac{5n+1}{64^n}D_n=\frac{8}{\sqrt3\pi}.
$$
In \cite{R}, Rogers showed the following identity by using very advanced and complicated method,
$$
\sum_{n=0}^\infty D_nu^n=\frac1{1-4u}\sum_{k=0}^\infty\binom{2k}k^2\binom{3k}k\l(\frac{u^2}{(1-4u)^3}\r)^k.
$$

Y.-P. Mu and Z.-W. Sun \cite{MS18} proved a congruence involving Domb numbers by telescoping method: For any prime $p>3$, we have the supercongruence
$$
\sum_{k=0}^{p-1}\frac{3k^2+k}{16^k}D_k\equiv-4p^4q_p(2)\pmod {p^5},
$$
where $q_p(a)$ denotes the Fermat quotient $(a^{p-1}-1)/p$.

\noindent Liu \cite{liud} proved some conjectures of Z.-W. Sun and Z.-H. Sun. For instance, Let $n$ be a positive integer. Then
$$
\frac1n\sum_{k=0}^{n-1}(2k+1)D_k8^{n-1-k}\ \ \ \mbox{and}\ \ \ \frac1n\sum_{k=0}^{n-1}(2k+1)D_k(-8)^{n-1-k}
$$
are all positive integers.

\noindent Z.-H. Sun gave the following congruence conjecture of the Domb numbers in \cite{SH2}:
\begin{conjecture}\label{Conj1} Let $p>3$ be a prime. Then
$$
D_{p-1}\equiv64^{p-1}-\frac{p^3}6B_{p-3}\pmod{p^4},
$$
\end{conjecture}
where $\{B_n\}$  are Bernoulli numbers given by
$$
 B_0=1,\ \ \ \sum_{k=0}^{n-1}\binom{n}{k}B_{k}=0\ \ (n\geq2).
$$
\noindent This conjecture was confirmed by the first author and J. Wang \cite{mw}. For more researches on Domb numbers, we refer the readers to (\cite{liud,sund} and so on).

\noindent In \cite{S14}, Z.-W. Sun proposed many congruence conjectures involving Domb numbers, for example \cite[Conjecture 5.2]{S14}:
\begin{conjecture}\label{Conj2} Let $p>3$ be a prime. We have
\begin{align}\label{x}
&\sum_{k=0}^{p-1}\frac{D_k}{4^k}\equiv\sum_{k=0}^{p-1}\frac{D_k}{16^k}\notag\\
&\equiv\begin{cases}4x^2-2p\ \pmod{p^2} &\tt{if}\ \textit{p}\equiv1\pmod 3 \&\ \textit{p}=\textit{x}^2+3\textit{y}^2\ (\textit{x},\textit{y}\in\mathbb{Z}),
\\0\ \pmod{p^2} &\tt{if}\ \textit{p}\equiv2\pmod3.\end{cases}
\end{align}
\end{conjecture}
Z-H. Sun \cite[Theorem 5.1]{S2020} proved this conjecture and proposed the following conjecture.
\begin{conjecture}\label{Conj3} Let $p>3$ be a prime. If $p\equiv1\pmod3$ and $p=x^2+3y^2$, then we have
$$
\sum_{k=0}^{p-1}\frac{D_k}{4^k}\equiv\sum_{k=0}^{p-1}\frac{D_k}{16^k}\equiv4x^2-2p-\frac{p^2}{4x^2}\pmod{p^3},
$$
and if $p\equiv2\pmod3$, then
$$
\sum_{k=0}^{p-1}\frac{D_k}{4^k}\equiv-2\sum_{k=0}^{p-1}\frac{D_k}{16^k}\equiv\frac{p^2}2\binom{\frac{p-1}2}{\frac{p-5}6}^{-2} \pmod{p^3},
$$
\end{conjecture}
\noindent In this paper, our main goal is to prove conjecture \ref{Conj3}.
 \begin{theorem}\label{Th1.1} Conjecture \ref{Conj3} is true.
\end{theorem}
Z.-W. Sun \cite{S14} also conjectured that If $p\equiv1\pmod3$, then
$$
\sum_{k=0}^{p-1}(3k+2)\frac{D_k}{4^k}\equiv\sum_{k=0}^{p-1}(3k+1)\frac{D_k}{{16}^k}\equiv0\pmod{p^2}.
$$
Our second goal is to prove the following stronger result and thus prove the above conjecture:
\begin{theorem}\label{Th1.2}If $p\equiv1\pmod3$, then
$$
\sum_{k=0}^{p-1}(3k+2)\frac{D_k}{4^k}\equiv2\sum_{k=0}^{p-1}(3k+1)\frac{D_k}{{16}^k}\equiv2p^2\binom{\frac{p-1}2}{\frac{p-1}6}^{-2}\pmod{p^3}.
$$
\end{theorem}
We also proof the following two conjectures of Z.-H. Sun in \cite[Conjecture 3.5, Conjecture 3.6]{s2111}: First, Sun defined that
$$
R_3(p)=\left(1+2p+\frac43(2^{p-1}-1)-\frac32(3^{p-1}-1)\right)\binom{\frac{p-1}2}{\lfloor p/6\rfloor}^2.
$$
\begin{theorem}\label{Th1.3} Let $p>3$ be a prime. Then
\begin{align*}
\sum_{k=0}^{p-1}k^2\frac{D_k}{4^k}\equiv\begin{cases}\frac{16}9x^2-\frac{8p}9-\frac{7p^2}{18x^2} \pmod{p^3} &\tt{if}\ \textit{p}=\textit{x}^2+3\textit{y}^2\equiv1\pmod 3,
\\ -\frac{20}9R_3(p) \pmod{p^2} &\tt{if}\ \textit{p}\equiv2\pmod3,\end{cases}\\
\sum_{k=0}^{p-1}k^2\frac{D_k}{16^k}\equiv\begin{cases}\frac{4}9x^2-\frac{2p}9-\frac{p^2}{18x^2} \pmod{p^3} &\tt{if}\ \textit{p}=\textit{x}^2+3\textit{y}^2\equiv1\pmod 3,
\\ \frac{4}9R_3(p) \pmod{p^2} &\tt{if}\ \textit{p}\equiv2\pmod3,\end{cases}
\end{align*}
and if $p\equiv2\pmod3$, 
\begin{align}
\sum_{k=0}^{p-1}k\frac{D_k}{4^k}\equiv-\sum_{k=0}^{p-1}k\frac{D_k}{16^k}\equiv\frac{4}3R_3(p) \pmod{p^2}\label{kdk4k}.
\end{align}
\end{theorem}
\begin{remark}
We also can prove the other two congruences in \cite[Conjecture 3.5, Conjecture 3.6]{s2111}, but the process of the proof is complex, so we will not give the details in this paper. Z.-H. Sun (private communication) conjectured (\ref{kdk4k}) which was not given public.
\end{remark}
We are going to prove Theorems \ref{Th1.1} and \ref{Th1.2} in Sections 2 and 3 respectively. Section 4 is devoted to proving Theorem \ref{Th1.3}. Our proofs make use of some combinatorial identities which can be found and proved by the package \texttt{Sigma} \cite{S} via the software \texttt{Mathematica}. We also rely on the $p$-adic Gamma function, Gamma function.
 \section{Proof of Theorem 1.1}
 \setcounter{lemma}{0}
\setcounter{theorem}{0}
\setcounter{corollary}{0}
\setcounter{remark}{0}
\setcounter{equation}{0}
\setcounter{conjecture}{0}
For a prime $p$, let  $\mathbb{Z}_p$ denote the ring of all $p$-adic integers and let $\mathbb{Z}_p^{\times}:=\{a\in\mathbb{Z}_p:\,a\text{ is prime to }p\}.$
For each $\alpha\in\mathbb{Z}_p$, define the $p$-adic order $\nu_p(\alpha):=\max\{n\in\mathbb{N}:\, p^n\mid \alpha\}$ and the $p$-adic norm $|\alpha|_p:=p^{-\nu_p(\alpha)}$. Define the $p$-adic gamma function $\Gamma_p(\cdot)$ by
$$
\Gamma_p(n)=(-1)^n\prod_{\substack{1\leq k<n\\ (k,p)=1}}k,\qquad n=1,2,3,\ldots,
$$
and
$$
\Gamma_p(\alpha)=\lim_{\substack{|\alpha-n|_p\to 0\\ n\in\mathbb{N}}}\Gamma_p(n),\qquad \alpha\in\\mathbb{Z}_p.
$$
In particular, we set $\Gamma_p(0)=1$. Following, we need to use the most basic properties of $\Gamma_p$, and all of them can be found in \cite{Murty02,Robert00}.
For example, we know that
\begin{equation}\label{Gammap}
\frac{\Gamma_p(x+1)}{\Gamma_p(x)}=\begin{cases}-x,&\text{if }|x|_p=1,\\
-1,&\text{if }|x|_p>1.
\end{cases}
\end{equation}
\begin{align}\label{Gammap1xx}
\Gamma_p(1-x)\Gamma_p(x)=(-1)^{a_0(x)},
\end{align}
where $a_0(x)\in\{1,2,\ldots,p\}$ such that $x\equiv a_0(x)\pmod p$. Among the properties we need here is the fact that for any positive integer $n$,
\begin{equation}\label{pn}
z_1\equiv z_2\pmod{p^n}\ \ \ \mbox{implies}\ \ \ \Gamma_p(z_1)\equiv\Gamma_p(z_2)\pmod{p^n}.
\end{equation}
Our proof of Theorem \ref{Th1.1} heavily relies on the following two transformation formulas due to Chan and Zudilin \cite{CZ} and Sun \cite{sund} respectively,
 \begin{align}\label{2.0}\sum_{k=0}^n\binom{n}k^2\binom{2k}k\binom{2n-2k}{n-k}=\sum_{k=0}^{n}(-1)^k\binom{n+2k}{3k}\binom{2k}k^2\binom{3k}k16^{n-k},
 \end{align}
 \begin{align}\label{2.1}\sum_{k=0}^n\binom{n}k^2\binom{2k}k\binom{2n-2k}{n-k}=\sum_{k=0}^{\lfloor n/2\rfloor}\binom{n+k}{3k}\binom{2k}k^2\binom{3k}k4^{n-2k}.
 \end{align}
 \begin{lemma}\label{sunh}{\rm (\cite{s2000,s2008})} Let $p>5$ be a prime. Then
\begin{gather*}
H_{p-1}^{(2)}\equiv0\pmod{p},\ \ H_{\frac{p-1}2}^{(2)}\equiv0\pmod{p},\ \ H_{p-1}\equiv0\pmod{p^2},\\
\frac15H_{\lfloor\frac{p}6\rfloor}^{(2)}\equiv H_{\lfloor\frac{p}3\rfloor}^{(2)}\equiv\frac12\left(\frac{p}3\right)B_{p-2}\left(\frac13\right)\pmod p,\\
H_{\lfloor\frac{p}6\rfloor}\equiv-2q_p(2)-\frac32q_p(3)+pq^2_p(2)+\frac{3p}4q^2_p(3)-\frac{5p}{12}\left(\frac{p}3\right)B_{p-2}\left(\frac13\right)\pmod{p^2},\\
H_{\lfloor\frac{p}3\rfloor}\equiv-\frac32q_p(3)+\frac{3p}4q^2_p(3)-\frac{p}6\left(\frac{p}3\right)B_{p-2}\left(\frac13\right)\pmod{p^2},\\
H_{\frac{p-1}2}\equiv-2q_p(2)+pq^2_p(2)\pmod{p^2},\ H_{\lfloor\frac{p}4\rfloor}^{(2)}\equiv(-1)^{\frac{p-1}2}4E_{p-3}\pmod p,\\
H_{\lfloor\frac{2p}3\rfloor}\equiv-\frac32q_p(3)+\frac{3p}4q^2_p(3)+\frac{p}3\left(\frac{p}3\right)B_{p-2}\left(\frac13\right)\pmod{p^2}.
\end{gather*}
\end{lemma}
\begin{lemma}\label{Lem2.2} Let $p>2$ be a prime and $p\equiv1\pmod3$. If $0\leq j\leq (p-1)/2$, then we have
$$
\binom{3j}j\binom{p+j}{3j+1}\equiv \frac{p}{3j+1}(1-pH_{2j}+pH_j)\pmod{p^3}.
$$
\end{lemma}
\begin{proof} If $0\leq j\leq (p-1)/2$ and $j\neq(p-1)/3$, then we have
\begin{align*}
\binom{3j}j\binom{p+j}{3j+1}&=\frac{(p+j)\cdots(p+1)p(p-1)\cdots(p-2j)}{j!(2j)!(3j+1)}\\
&\equiv\frac{pj!(1+pH_j)(-1)^{2j}(2j)!(1-pH_{2j})}{j!(2j)!(3j+1)}\\
&\equiv\frac{p}{3j+1}(1-pH_{2j}+pH_j)\pmod{p^3}.
\end{align*}
If $j=(p-1)/3$, then by Lemma \ref{sunh}, we have
\begin{align*}
&\binom{p-1}{\frac{p-1}3}\binom{p+\frac{p-1}3}{\frac{p-1}3}\\
\equiv&\left(1-pH_{\frac{p-1}3}+\frac{p^2}2(H_{\frac{p-1}3}^2-H_{\frac{p-1}3}^{(2)})\right)\left(1+pH_{\frac{p-1}3}+\frac{p^2}2(H_{\frac{p-1}3}^2-H_{\frac{p-1}3}^{(2)})\right)\\
\equiv&1-p^2H_{\frac{p-1}3}^{(2)}\equiv1-\frac{p^2}2\left(\frac{p}3\right)B_{p-2}\left(\frac13\right)\pmod{p^3}
\end{align*}
and
$$
1-pH_{\frac{2p-2}3}+pH_{\frac{p-1}3}\equiv1-\frac{p^2}2\left(\frac{p}3\right)B_{p-2}\left(\frac13\right)\pmod{p^3}.
$$
Now the proof of Lemma \ref{Lem2.2} is complete.
\end{proof}
\begin{lemma}\label{mpt} Let $p>3$ be a prime. For any $p$-adic integer $t$, we have
\begin{equation*}
 \binom{\frac{2p-2}3+pt}{\frac{p-1}2}\equiv\binom{\frac{2p-2}3}{\frac{p-1}2}\left(1+pt(H_{\frac{2p-2}3}-H_{\frac{p-1}6})\right)\pmod{p^2}.
 \end{equation*}
 \end{lemma}
\begin{proof} Set $m=(2p-2)/3$. It is easy to check that
\begin{align*}
\binom{m+pt}{(p-1)/2}&=\frac{(m+pt)\cdots(m+pt-(p-1)/2+1)}{((p-1)/2)!}\\
&\equiv\frac{m\cdots(m-(p-1)/2+1)}{((p-1)/2)!}(1+pt(H_m-H_{m-(p-1)/2})\\
&=\binom{m}{(p-1)/2}(1+pt(H_m-H_{m-(p-1)/2})\pmod{p^2}.
\end{align*}
So Lemma \ref{mpt} is finished.
\end{proof}
 \medskip
\noindent{\it Proof of Theorem 1.1}. Firstly, we prove the first congruence.

\noindent{\it Case} $p\equiv1\pmod3$. With the help of (\ref{2.1}), we have
\begin{align}\label{3k4f}
\sum_{k=0}^{p-1}\frac{D_k}{4^k}&=\sum_{k=0}^{p-1}\frac{1}{4^k}\sum_{j=0}^{\lfloor k/2\rfloor}\binom{k+j}{3j}\binom{2j}j^2\binom{3j}j4^{k-2j}\notag\\
&=\sum_{j=0}^{(p-1)/2}\frac{\binom{2j}j^2\binom{3j}j}{16^j}\sum_{k=2j}^{p-1}\binom{k+j}{3j}.
\end{align}
By loading the package \texttt{Sigma} in the software \texttt{Mathematica}, we have the following identity:
$$
\sum_{k=2j}^{n-1}\binom{k+j}{3j}=\binom{n+j}{3j+1}.$$
Thus, replacing $n$ by $p$ in the above identity and then substitute it into (\ref{3k4f}), we have
\begin{equation*}
\sum_{k=0}^{p-1}\frac{D_k}{4^k}=\sum_{j=0}^{(p-1)/2}\frac{\binom{2j}j^2\binom{3j}j}{16^j}\binom{p+j}{3j+1}.
\end{equation*}
Hence we immediately obtain the following result by Lemma \ref{Lem2.2},
\begin{equation}\label{main1}
\sum_{k=0}^{p-1}\frac{D_k}{4^k}\equiv p\sum_{j=0}^{(p-1)/2}\frac{\binom{2j}j^2}{16^j}\frac{1-pH_{2j}+pH_j}{3j+1}\pmod{p^3}.
\end{equation}
Since $\binom{2k}k^2/16^k\equiv\binom{(p-1)/2}k\binom{(p-1)/2+k}k(-1)^k\pmod{p^2}$ for each $0\leq k\leq (p-1)/2$ , it is easy to verify that
\begin{align*}
\sum_{j=0}^{\frac{p-1}2}\frac{\binom{2j}j^2}{16^j}\frac{H_{j}-H_{2j}}{3j+1}\equiv \sum_{j=0}^{\frac{p-1}2}\frac{\binom{\frac{p-1}2}j\binom{\frac{p-1}2+j}{j}(-1)^j(H_{j}-H_{2j})}{3j+1}\pmod{p}.
\end{align*}
By \texttt{Sigma}, we found the following identity:
\begin{align}\label{equat}
\sum_{k=0}^n\frac{\binom{n}k\binom{n+k}k(-1)^k(H_{k}-H_{2k})}{3k+1}=\frac{1}{3n+1}\prod_{k=1}^n\frac{3k-1}{3k-2}\sum_{k=1}^n\frac1k\prod_{j=1}^k\frac{3j-2}{3j-1}.
\end{align}
In view of \cite{ml}, we have
\begin{gather*}
\sum_{k=1}^{\frac{p-1}3}\frac{4^k}{k\binom{2k}k}\equiv-2+\frac{2}{\binom{\frac{p-1}2}{\frac{p-1}3}}\equiv-2+\frac1x\pmod p,\\
3\sum_{j=1}^{\frac{p-1}3}\frac{4^j }{(3j-1)\binom{2j}{j}}\equiv-2+\frac1x+\frac13\binom{\frac{p-1}2}{\frac{p-1}3}\sum_{k=1}^{\frac{p-1}3}\frac{4^k}{k^2\binom{2k}k}\pmod{p}.
\end{gather*}
So by \cite[Lemma 3.1]{sijnt}, we have
\begin{align*}
&\sum_{k=1}^{\frac{p-1}2}\frac1k\prod_{j=1}^k\frac{3j-2}{3j-1}=\sum_{k=1}^{\frac{p-1}2}\frac{\binom{-1/3}k}{k\binom{-2/3}k}\equiv\frac{p}3\sum_{k=1}^{\frac{p-1}2}\frac{(-1)^k}{k^2\binom{-2/3}k}-\sum_{k=1}^{\frac{p-1}3}\frac{3}{3k-1}\\
&-3p\sum_{k=1}^{\frac{p-1}3}\frac{1}{(3k-1)^2}-\frac{p(-1)^{\frac{p-1}2}}{3\binom{\frac{2p-2}3}{\frac{p-1}2}}\binom{\frac{p-1}2}{\frac{p-1}3}\sum_{k=1}^{\frac{p-1}3}\frac{4^k}{k^2\binom{2k}k}\pmod{p^2}.
\end{align*}
It is easy to check that
\begin{align*}
\sum_{k=1}^{\frac{p-1}3}\frac{4^k}{k^2\binom{2k}k}\equiv\sum_{k=1}^{\frac{p-1}3}\frac{(-1)^k}{k^2\binom{\frac{p-1}2}k}\equiv2\sum_{k=0}^{\frac{p-4}3}\frac{(-1)^k}{(k+1)\binom{\frac{p-3}2}k}\pmod p.
\end{align*}
And by \cite[(6)]{swz}, we have
\begin{equation}\label{cyid}
\frac1{\binom{n+1+k}k}=(n+1)\sum_{r=0}^n\binom{n}r\frac{(-1)^r}{k+r+1}.
\end{equation}
Hence, setting $n=\frac{p-1}2$ in the above identity, we have
\begin{align*}
&2\sum_{k=0}^{\frac{p-4}3}\frac{(-1)^k}{(k+1)\binom{\frac{p-3}2}k}\equiv2\sum_{k=0}^{\frac{p-4}3}\frac{1}{(k+1)\binom{\frac{p+1}2}k}\\
&\equiv\sum_{k=0}^{\frac{p-4}3}\frac1{k+1}\sum_{r=0}^{\frac{p-1}2}\binom{\frac{p-1}2}r\frac{(-1)^r}{k+1+r}=\sum_{k=1}^{\frac{p-1}3}\frac1{k}\sum_{r=0}^{\frac{p-1}2}\binom{\frac{p-1}2}r\frac{(-1)^r}{k+r}\\
&=H_{\frac{p-1}3}^{(2)}+\sum_{r=1}^{\frac{p-1}2}\binom{\frac{p-1}2}r\frac{(-1)^r}r\sum_{k=1}^{\frac{p-1}3}\left(\frac1{k}-\frac1{k+r}\right)\pmod p.
\end{align*}
It is easy to obtain that
$$\sum_{k=1}^{\frac{p-1}3}\left(\frac1{k}-\frac1{k+r}\right)\equiv-\sum_{k=1}^r\frac1{k(3k-1)}\pmod{p}.$$
And by \texttt{Sigma}, we find the following identity which can be proved by induction on $n$:
$$
\sum_{r=1}^n\binom{n}r\frac{(-1)^r}r\sum_{k=1}^r\frac1{k(3k-1)}=H_n^{(2)}-\sum_{k=1}^n\frac{(-1)^k}{k^2\binom{-2/3}k}.
$$
So we have
\begin{align*}
\sum_{k=1}^{\frac{p-1}3}\frac{4^k}{k^2\binom{2k}k}&\equiv2\sum_{k=0}^{\frac{p-4}3}\frac{(-1)^k}{(k+1)\binom{\frac{p-3}2}k}\\
&\equiv H_{\frac{p-1}3}^{(2)}-H_{\frac{p-1}2}^{(2)}+\sum_{k=1}^{\frac{p-1}2}\frac{(-1)^k}{k^2\binom{-2/3}k}\pmod p.
\end{align*}
Then by \cite[Theorem4.12]{YMK} and Lemma \ref{sunh}, we have
\begin{align*}
&\sum_{k=1}^{\frac{p-1}2}\frac1k\prod_{j=1}^k\frac{3j-2}{3j-1}\equiv-\sum_{k=1}^{\frac{p-1}3}\frac{3}{3k-1}-3p\sum_{k=1}^{\frac{p-1}3}\frac{1}{(3k-1)^2}-\frac{p}3H_{\frac{p-1}3}^{(2)}\\
&\equiv3\sum_{k=1}^{\frac{p-1}3}\frac1{3k-1}-\frac{p}3H_{\frac{p-1}3}^{(2)}\equiv 0\pmod{p^2},
\end{align*}
where we used \cite[Lemma 2.3, Lemma 2.6, Lemma 2.7]{s2008}, which help us deduce that
\begin{align*}
&\sum_{k=1}^{\frac{p-1}3}\frac1{3k-1}=\sum_{k=1\atop k\equiv2\pmod3}^{p-1}\frac1k\\
&\equiv\frac{B_{\varphi(p^3)}\left(\frac13\right)-B_{\varphi(p^3)}\left(\frac23\right)}{3\varphi(p^3)}+\frac{p}9\left(\frac{B_{2p-3}(\frac13)}{2p-3}-2\frac{B_{p-2}(\frac13)}{p-2}\right)\\
&=0+\frac{p}9\left(\frac{B_{p-1+p-2}(\frac13)}{p-1+p-2}-2\frac{B_{p-2}(\frac13)}{p-2}\right)\equiv-\frac{p}9\frac{B_{p-2}(\frac13)}{p-2}\\
&\equiv\frac{p}{18}B_{p-2}\left(\frac13\right)\pmod{p^2}.
\end{align*}
So it is easy to see that
$$
\frac{2}{3p-1}\prod_{k=1}^{\frac{p-1}2}\frac{3k-1}{3k-2}\sum_{k=1}^{\frac{p-1}2}\frac1k\prod_{j=1}^k\frac{3j-2}{3j-1}\equiv0\pmod p.
$$
And hence,
\begin{align}\label{0p}
\sum_{j=0}^{\frac{p-1}2}\frac{\binom{2j}j^2}{16^j}\frac{H_{j}-H_{2j}}{3j+1}\equiv \sum_{j=0}^{\frac{p-1}2}\frac{\binom{\frac{p-1}2}j\binom{\frac{p-1}2+j}{j}(-1)^j(H_{j}-H_{2j})}{3j+1}\equiv0\pmod{p}.
\end{align}
In view of \cite{mp3}, we have
$$
p\sum_{j=0}^{(p-1)/2}\frac{\binom{2j}j^2}{16^j}\frac{1}{3j+1}\equiv4x^2-2p-\frac{p^2}{4x^2}\pmod{p^3}.
$$
So the case $p\equiv1\pmod3$ is finished.

\noindent{\it Case} $p\equiv2\pmod3$. In the same way of above, we have
\begin{align*}
&\sum_{k=0}^{p-1}\frac{D_k}{4^k}\equiv p\sum_{j=0}^{(p-1)/2}\frac{\binom{2j}j^2}{16^j}\frac{1-pH_{2j}+pH_j}{3j+1}\\
&\equiv p\sum_{j=0}^{\frac{p-1}2}\frac{\binom{\frac{p-1}2}j\binom{\frac{p-1}2+j}j(-1)^j}{3j+1}+p^2\sum_{j=0}^{\frac{p-1}2}\frac{\binom{\frac{p-1}2}j\binom{\frac{p-1}2+j}j(-1)^j(H_{j}-H_{2j})}{3j+1}\\
&=\frac{2p}{3p-1}\frac{(2/3)_{(p-1)/2}}{(1/3)_{(p-1)/2}}+\frac{2p^2}{3p-1}\frac{(2/3)_{(p-1)/2}}{(1/3)_{(p-1)/2}}\sum_{k=1}^{\frac{p-1}2}\frac{(\frac13)_k}{k(\frac23)_k}\pmod{p^3},
\end{align*}
where we used the following identity and (\ref{equat}):
$$
\sum_{k=0}^n\binom nk\binom{n+k}k\frac{(-1)^k}{3k+1}=\frac1{3n+1}\prod_{k=1}^n\frac{3k-1}{3k-2}.
$$
It is easy to see that
\begin{align}\label{similar}
&p\sum_{k=1}^{\frac{p-1}2}\frac{(\frac13)_k}{k(\frac23)_k}\equiv p\sum_{k=\frac{p+1}3}^{\frac{p-1}2}\frac{(\frac13)_k}{k(\frac23)_k}=p\sum_{k=1}^{\frac{p+1}6}\frac{(\frac13)_{k+\frac{p-2}3}}{(k+\frac{p-2}3)(\frac23)_{k+\frac{p-2}3}}\notag\\
&\equiv3\sum_{k=1}^{\frac{p+1}6}\frac{\binom{-1/3}{k+\frac{p-2}3}(-1)^{k+\frac{p-2}3}(k+\frac{p-5}3)!}{\frac23\cdots(\frac{p}3-1)(\frac{p}3+1)\cdots(\frac{p}3+k-1)}\notag\\
&\equiv3\sum_{k=1}^{\frac{p+1}6}\binom{-1/3}{k+\frac{p-2}3}\binom{k+\frac{p-5}3}{k-1}(-1)^k\equiv-3\sum_{k=1}^{\frac{p+1}6}\binom{-1/3}{k+\frac{p-2}3}\binom{-1/3}{k-1}\notag\\
&\equiv-3\sum_{k=1}^{\frac{p+1}6}\binom{\frac{2p-1}3}{\frac{p+1}3-k}\binom{-1/3}{k-1}\equiv-3\sum_{k=1}^{\frac{p+1}6}\binom{-1/3}{\frac{p+1}3-k}\binom{-1/3}{k-1}\pmod p.
\end{align}
We can find and prove the following identity by \texttt{Sigma}:
$$
\sum_{k=1}^{n}\binom{-1/3}{2n-k}\binom{-1/3}{k-1}=-\frac{3n}{6n-1}\prod_{k=1}^n\frac{(3k-2)(6k-1)}{9k(2k-1)}.
$$
So by substituting $n=(p+1)/6$ into the above identity and (\ref{pn}), we have
\begin{align*}
&p\sum_{k=1}^{\frac{p-1}2}\frac{(\frac13)_k}{k(\frac23)_k}\equiv\frac{3}{2p}\frac{(\frac13)_{\frac{p+1}6}(\frac56)_{\frac{p+1}6}}{(1)_{\frac{p+1}6}(\frac12)_{\frac{p+1}6}}\equiv\frac32\frac{(\frac56)_{\frac{p-5}6}}{(1)_{\frac{p-5}6}}\frac{(\frac13)_{\frac{p+1}6}}{(\frac12)_{\frac{p+1}6}}\\
&\equiv\frac{3(-1)^{\frac{p-5}6}}2\frac{\Gamma_p(\frac{p}6+\frac12)\Gamma_p(\frac12)}{\Gamma_p(\frac{p}6+\frac23)\Gamma_p(\frac13)}\equiv\frac{3(-1)^{\frac{p-5}6}}2\frac{\Gamma_p(\frac12)\Gamma_p(\frac12)}{\Gamma_p(\frac23)\Gamma_p(\frac13)}\pmod p.
\end{align*}
Hence, by (\ref{Gammap1xx}), we have
\begin{equation}\label{zhong}
p\sum_{k=1}^{\frac{p-1}2}\frac{(\frac13)_k}{k(\frac23)_k}\equiv\frac{3(-1)^{\frac{p-5}6}}2(-1)^{\frac{p+1}2}(-1)^{\frac{p+1}3}=-\frac32\pmod p.
\end{equation}
So
\begin{equation}\label{zhu11}
\sum_{k=0}^{p-1}\frac{D_k}{4^k}\equiv-\frac12\frac{2p}{3p-1}\frac{(2/3)_{(p-1)/2}}{(1/3)_{(p-1)/2}}\pmod{p^3}.
\end{equation}
And it is easy to see that
\begin{align*}
&\frac{2p}{3p-1}\frac{(2/3)_{(p-1)/2}}{(1/3)_{(p-1)/2}}\equiv-\frac{2p^2}3\frac{\frac23\cdots(\frac{p}3-1)(\frac{p}3+1)\cdots(\frac{p}3+\frac{p-5}6)}{(\frac13)_{\frac{p-1}2}}\notag\\
&\equiv\frac{2p^2}{\binom{\frac{p-1}2}{\frac{p-5}6}}\frac{(1)_{\frac{p-1}2}}{(\frac13)_{\frac{p-1}2}}=-\frac{2p^2}{\binom{\frac{p-1}2}{\frac{p-5}6}}\frac{\Gamma_p(\frac{p+1}2)\Gamma_p(\frac13)}{\Gamma_p(\frac{p}2-\frac16)}\pmod{p^3}.
\end{align*}
Then by (\ref{Gammap}), (\ref{Gammap1xx}) and (\ref{pn}) we have
\begin{align*}
&\frac{\Gamma_p(\frac{p+1}2)\Gamma_p(\frac13)}{\Gamma_p(\frac{p}2-\frac16)}\equiv(-1)^{\frac{5p-1}6}\Gamma_p(\frac12)\Gamma_p(\frac13)\Gamma_p(\frac76)=\frac{(-1)^{\frac{5p+5}6}}6\Gamma_p(\frac12)\Gamma_p(\frac13)\Gamma_p(\frac16)\\
&=\frac16\frac{\Gamma_p(\frac13)\Gamma_p(\frac16)}{\Gamma_p(\frac12)}\equiv\frac16\frac{\Gamma_p(\frac{p+1}3)\Gamma_p(\frac{p+1}6)}{\Gamma_p(\frac{p+1}2)}=\frac16\frac{\Gamma(\frac{p+1}3)\Gamma(\frac{p+1}6)}{\Gamma(\frac{p+1}2)}\equiv\frac12\frac{1}{\binom{\frac{p-1}2}{\frac{p-5}6}}\pmod p.
\end{align*}
Thus,
\begin{equation}\label{zhu12}
\frac{2p}{3p-1}\frac{(2/3)_{(p-1)/2}}{(1/3)_{(p-1)/2}}\equiv-p^2\binom{\frac{p-1}2}{\frac{p-5}6}^{-2}\pmod{p^3}.
\end{equation}
This, with (\ref{zhu11}) yields that
$$
\sum_{k=0}^{p-1}\frac{D_k}{4^k}\equiv\frac{p^2}2\binom{\frac{p-1}2}{\frac{p-5}6}^{-2}\pmod{p^3}.
$$
Therefore we obtain the desired result
 \begin{align*}
&\sum_{k=0}^{p-1}\frac{D_k}{4^k}\equiv\\
&\begin{cases}4x^2-2p-\frac{p^2}{4x^2}\ \pmod{p^3} &\tt{if}\ \textit{p}\equiv1\pmod 3\ \&\ \textit{p}=\textit{x}^2+3\textit{y}^2\ (\textit{x},\textit{y}\in\mathbb{Z}),
\\ \frac{p^2}2\binom{\frac{p-1}2}{\frac{p-5}6}^{-2}\ \pmod{p^3} &\tt{if}\ \textit{p}\equiv2\pmod3.\end{cases}
\end{align*}
On the other hand, in view of \cite[(5.5)]{SH2}, we have
\begin{align*}
\sum_{k=0}^{p-1}\frac{D_k}{16^k}\equiv\sum_{k=0}^{p-1}\frac{\binom{2k}k^2}{16^k}\frac{p}{3k+1}(1+pH_{2k}-pH_k)\pmod{p^3}.
\end{align*}
This, with (\ref{0p}) yields that if $p\equiv1\pmod3$ and $p=x^2+3y^2$,
$$
\sum_{k=0}^{p-1}\frac{D_k}{16^k}\equiv p\sum_{j=0}^{(p-1)/2}\frac{\binom{2j}j^2}{16^j}\frac{1}{3j+1}\equiv4x^2-2p-\frac{p^2}{4x^2}\pmod{p^3}
$$
and if $p\equiv2\pmod3$,
\begin{align}\label{zhu}
&\sum_{k=0}^{p-1}\frac{D_k}{16^k}-Y_1\equiv p\sum_{j=0}^{(p-1)/2}\frac{\binom{2j}j^2}{16^j}\frac{1+pH_{2j}-pH_j}{3j+1}\notag\\
&\equiv p\sum_{j=0}^{\frac{p-1}2}\frac{\binom{\frac{p-1}2}j\binom{\frac{p-1}2+j}j(-1)^j}{3j+1}-p^2\sum_{j=0}^{\frac{p-1}2}\frac{\binom{\frac{p-1}2}j\binom{\frac{p-1}2+j}j(-1)^j(H_{j}-H_{2j})}{3j+1}\notag\\
&=\frac{2p}{3p-1}\frac{(2/3)_{(p-1)/2}}{(1/3)_{(p-1)/2}}-\frac{2p^2}{3p-1}\frac{(2/3)_{(p-1)/2}}{(1/3)_{(p-1)/2}}\sum_{k=1}^{\frac{p-1}2}\frac{(\frac13)_k}{k(\frac23)_k}\notag\\
&\equiv\frac52\frac{2p}{3p-1}\frac{(2/3)_{(p-1)/2}}{(1/3)_{(p-1)/2}}\pmod{p^3},
\end{align}
where
$$
Y_1=\frac12\frac{\binom{\frac{4p-2}3}{\frac{2p-1}3}^2}{16^{\frac{2p-1}3}}(1+pH_{\frac{4p-2}3}-pH_{\frac{2p-1}3})=\frac12\binom{-1/2}{\frac{2p-1}3}^2(1+pH_{\frac{4p-2}3}-pH_{\frac{2p-1}3}).
$$
It is easy to see that
\begin{align*}
\binom{-1/2}{\frac{2p-1}3}^2=\frac{\binom{-1/2}{\frac{p-1}2}^2\binom{-p/2}{\frac{p+1}6}^2}{\binom{\frac{2p-1}3}{\frac{p-1}2}^2}\equiv\frac{9p^2}{\binom{\frac{2p-1}3}{\frac{p-1}2}^2}\pmod{p^3}.
\end{align*}
And by (\ref{Gammap}), (\ref{Gammap1xx}) and (\ref{pn}), we have
\begin{align*}
&\binom{\frac{2p-1}3}{\frac{p-1}2}^2=\frac{\Gamma(\frac{2p+2}3)^2}{\Gamma(\frac{p+1}2)^2\Gamma(\frac{p+7}6)^2}=\frac{\Gamma_p(\frac{2p+2}3)^2}{\Gamma_p(\frac{p+1}2)^2\Gamma_p(\frac{p+7}6)^2}\equiv\frac{\Gamma_p(\frac23)^2}{\Gamma_p(\frac12)^2\Gamma_p(\frac76)^2}\\
&=\frac{36\Gamma_p(\frac12)^2}{\Gamma_p(\frac13)^2\Gamma_p(\frac16)^2}\equiv\frac{36\Gamma_p(\frac{p+1}2)^2}{\Gamma_p(\frac{p+1}3)^2\Gamma_p(\frac{p+1}6)^2}=\frac{36(\frac{p-1}2)!^2}{(\frac{p-2}3)!^2(\frac{p-5}6)!^2}\equiv4\binom{\frac{p-1}2}{\frac{p-5}6}^2\pmod p.
\end{align*}
It is easy to see that
$$
1+pH_{\frac{4p-2}3}-pH_{\frac{2p-1}3}\equiv2\pmod p.
$$
These yield that
$$
Y_1\equiv\frac{9p^2}{4}\binom{\frac{p-1}2}{\frac{p-5}6}^{-2}\pmod {p^3}.
$$
This, with (\ref{zhu12}) and (\ref{zhu}) yields the desired result
\begin{align*}
\sum_{k=0}^{p-1}\frac{D_k}{16^k}\equiv-\frac{p^2}4\binom{\frac{p-1}2}{\frac{p-5}6}^{-2}\pmod{p^3}.
\end{align*}
Now we finish the proof of Theorem \ref{Th1.1}.
\section{Proof of Theorem \ref{Th1.2}}
\setcounter{lemma}{0}
\setcounter{theorem}{0}
\setcounter{corollary}{0}
\setcounter{remark}{0}
\setcounter{equation}{0}
\setcounter{conjecture}{0}
\begin{lemma}\label{p2j} Let $p>2$ be a prime. If $0\leq j\leq (p-1)/2$, then we have
$$
(3j+1)\binom{3j}j\binom{p+2j}{3j+1}\equiv p(-1)^j(1+pH_{2j}-pH_j)\pmod{p^3}.
$$
If $(p+1)/2\leq j\leq p-1$, then
$$
(3j+1)\binom{3j}j\binom{p+2j}{3j+1}\equiv 2p^2(-1)^j(H_{2j}-H_j)\pmod{p^3}.
$$
\end{lemma}
\begin{proof}If $0\leq j\leq (p-1)/2$, then we have
\begin{align*}
(3j+1)\binom{3j}j\binom{p+2j}{3j+1}&=\frac{(p+2j)\cdots(p+1)p(p-1)\cdots(p-j)}{j!(2j)!}\\
&\equiv\frac{p(2j)!(1+pH_{2j})(-1)^{j}(j)!(1-pH_j)}{j!(2j)!}\\
&=p(-1)^j(1+pH_{2j}-pH_{j})\pmod{p^3}.
\end{align*}
If $(p+1)/2\leq j\leq p-1$, then by Lemma \ref{sunh}, we have
\begin{align*}
&(3j+1)\binom{3j}j\binom{p+2j}{3j+1}\\
&=\frac{(p+2j)\cdots(2p+1)(2p)(2p-1)\cdots(p+1)p(p-1)\cdots(p-j)}{j!(2j)!}\\
&\equiv\frac{2p^2(2j)\cdots(p+1)\left(1+p\sum_{k=p+1}^{2j}\frac1k\right)(p-1)!(-1)^{j}(j)!(1-pH_j)}{j!(2j)!}\\
&=2p(-1)^j\left(1+p\sum_{k=p+1}^{2j}\frac1k\right)(1-pH_j)\equiv2p(-1)^jpH_{2j}(1-pH_j)\\
&\equiv2p^2(-1)^j(H_{2j}-pH_{2j}H_j)\equiv2p^2(-1)^j(H_{2j}-H_j)\pmod{p^3}.
\end{align*}
Now the proof of Lemma \ref{p2j} is complete.
\end{proof}
\noindent{\it Proof of Theorem \ref{Th1.2}}. Similarly, by (\ref{2.1}), we have
\begin{align}\label{dk4}
\sum_{k=0}^{p-1}(3k+2)\frac{D_k}{4^k}&=\sum_{k=0}^{p-1}\frac{3k+2}{4^k}\sum_{j=0}^{\lfloor k/2\rfloor}\binom{k+j}{3j}\binom{2j}j^2\binom{3j}j4^{k-2j}\notag\\
&=\sum_{j=0}^{(p-1)/2}\frac{\binom{2j}j^2\binom{3j}j}{16^j}\sum_{k=2j}^{p-1}(3k+2)\binom{k+j}{3j}.
\end{align}
By loading the package \texttt{Sigma} in the software \texttt{Mathematica}, we have the following identity:
$$
\sum_{k=2j}^{n-1}(3k+2)\binom{k+j}{3j}=\frac{(3n+1)(3j+1)}{3j+2}\binom{n+j}{3j+1}.$$
Thus, replacing $n$ by $p$ in the above identity and then substitute it into (\ref{dk4}), we have
\begin{equation*}
\sum_{k=0}^{p-1}(3k+2)\frac{D_k}{4^k}=\sum_{j=0}^{(p-1)/2}\frac{\binom{2j}j^2\binom{3j}j}{16^j}\frac{(3p+1)(3j+1)}{3j+2}\binom{p+j}{3j+1}.
\end{equation*}
Combining Lemma \ref{Lem2.2} we can obtain that for any $0\leq j\leq (p-1)/2$,
$$
(3j+1)\binom{3j}j\binom{p+j}{3j+1}\equiv p(1-pH_{2j}+pH_j)\pmod{p^3}.
$$
Since $p\equiv1\pmod3$, so we have
\begin{align*}
&\frac1{3p+1}\sum_{k=0}^{p-1}(3k+2)\frac{D_k}{4^k}\equiv p\sum_{j=0}^{(p-1)/2}\frac{\binom{2j}j^2}{16^j}\frac{1-pH_{2j}+pH_j}{3j+2}\\
\equiv& p\sum_{j=0}^{\frac{p-1}2}\frac{\binom{\frac{p-1}2}j\binom{\frac{p-1}2+j}j(-1)^j}{3j+2}-p^2\sum_{j=0}^{\frac{p-1}2}\frac{\binom{\frac{p-1}2}j\binom{\frac{p-1}2+j}j(-1)^j(H_{2j}-H_j)}{3j+2}\pmod{p^3}.
\end{align*}
By the \texttt{Sigma} again, we find the following two identities:
\begin{gather*}
\sum_{k=0}^n\binom{n}k\binom{n+k}k\frac{(-1)^k}{3k+2}=\frac1{3n+2}\prod_{k=1}^n\frac{3k-2}{3k-1},\\
\sum_{k=0}^n\binom{n}k\binom{n+k}k\frac{(-1)^k(H_{2j}-H_j)}{3k+2}=-\frac1{3n+2}\prod_{k=1}^n\frac{3k-2}{3k-1}\sum_{k=1}^n\frac1k\prod_{j=1}^k\frac{3j-1}{3j-2}.
\end{gather*}
Hence
\begin{align}\label{xuyao}
\sum_{k=0}^{p-1}(3k+2)\frac{D_k}{4^k}\equiv2p\prod_{k=1}^{\frac{p-1}2}\frac{3k-2}{3k-1}\left(1+p\sum_{k=1}^n\frac1k\prod_{j=1}^k\frac{3j-1}{3j-2}\right)\pmod{p^3}.
\end{align}
And it is easy to see that
\begin{align*}
&2p\frac{(1/3)_{(p-1)/2}}{(2/3)_{(p-1)/2}}\equiv\frac{2p^2}3\frac{\frac13\cdots(\frac{p}3-1)(\frac{p}3+1)\cdots(\frac{p}3+\frac{p-7}6)}{(\frac23)_{\frac{p-1}2}}\notag\\
&\equiv\frac{-4p^2}{\binom{\frac{p-1}2}{\frac{p-1}6}}\frac{(1)_{\frac{p-1}2}}{(\frac23)_{\frac{p-1}2}}=\frac{4p^2}{\binom{\frac{p-1}2}{\frac{p-1}6}}\frac{\Gamma_p(\frac{p+1}2)\Gamma_p(\frac23)}{\Gamma_p(\frac{p}2+\frac16)}\pmod{p^3}.
\end{align*}
Then by (\ref{Gammap}), (\ref{Gammap1xx}) and (\ref{pn}) we have
\begin{align*}
&\frac{\Gamma_p(\frac{p+1}2)\Gamma_p(\frac23)}{\Gamma_p(\frac{p}2+\frac16)}\equiv\frac{\Gamma_p(\frac12)\Gamma_p(\frac23)}{\Gamma_p(\frac16)}=(-1)^{\frac{p-1}2}\frac{\Gamma_p(\frac56)\Gamma_p(\frac12)}{\Gamma_p(\frac13)}\\
&\equiv(-1)^{\frac{p-1}2}\frac{\Gamma_p(\frac{p+5}6)\Gamma_p(\frac{p+1}2)}{\Gamma_p(\frac{2p+1}3)}=(-1)^{\frac{p+5}6}\frac{\Gamma(\frac{p+5}6)\Gamma(\frac{p+1}2)}{\Gamma(\frac{2p+1}3)}\\
&\equiv(-1)^{\frac{p+5}6}\frac{1}{\binom{\frac{2p-2}3}{\frac{p-1}6}}\equiv-\frac{1}{\binom{\frac{p-1}2}{\frac{p-1}6}}\pmod p.
\end{align*}
Thus,
\begin{equation}\label{zhu13}
2p\frac{(1/3)_{(p-1)/2}}{(2/3)_{(p-1)/2}}\equiv-4p^2\binom{\frac{p-1}2}{\frac{p-1}6}^{-2}\pmod{p^3}.
\end{equation}
By similar manipulation as (\ref{similar}), we have
\begin{align*}
p\sum_{k=1}^{\frac{p-1}2}\frac{(\frac23)_k}{k(\frac13)_k}\equiv -3\sum_{k=1}^{\frac{p-1}6}\binom{-2/3}{\frac{p-1}3-k}\binom{-2/3}{k-1}\pmod p.
\end{align*}
We can find and prove the following identity by \texttt{Sigma}:
$$
\sum_{k=1}^{n}\binom{-2/3}{2n-k}\binom{-2/3}{k-1}=-3n\prod_{k=1}^n\frac{(3k-1)(6k-5)}{9k(2k-1)}.
$$
So by substituting $n=(p-1)/6$ into the above identity and (\ref{pn}), we have
\begin{align*}
&p\sum_{k=1}^{\frac{p-1}2}\frac{(\frac23)_k}{k(\frac13)_k}\equiv-\frac{3}{2}\frac{(\frac23)_{\frac{p-1}6}(\frac16)_{\frac{p-1}6}}{(1)_{\frac{p-1}6}(\frac12)_{\frac{p-1}6}}\equiv\frac{3(-1)^{\frac{p+5}6}}2\frac{\Gamma_p(\frac{p}6+\frac12)\Gamma_p(\frac12)}{\Gamma_p(\frac{p}6+\frac13)\Gamma_p(\frac23)}\\
&\equiv\frac{3(-1)^{\frac{p+5}6}}2\frac{\Gamma_p(\frac12)\Gamma_p(\frac12)}{\Gamma_p(\frac23)\Gamma_p(\frac13)}\pmod p.
\end{align*}
Hence, by (\ref{Gammap1xx}), we have
\begin{equation}\label{zhong1}
p\sum_{k=1}^{\frac{p-1}2}\frac{(\frac23)_k}{k(\frac13)_k}\equiv\frac{3(-1)^{\frac{p+5}6}}2(-1)^{\frac{p+1}2}(-1)^{\frac{2p+1}3}=-\frac32\pmod p.
\end{equation}
This, with (\ref{xuyao}) and (\ref{zhu13}) yields that
\begin{align*}
\sum_{k=0}^{p-1}(3k+2)\frac{D_k}{4^k}\equiv2p^2\binom{\frac{p-1}2}{\frac{p-1}6}^{-2}\pmod{p^3}.
\end{align*}
In the same way, by (\ref{2.0}), we have
\begin{align*}
\sum_{k=0}^{p-1}(3k+1)\frac{D_k}{16^k}&=\sum_{k=0}^{p-1}\frac{3k+1}{16^k}\sum_{j=0}^{k}(-1)^j\binom{k+2j}{3j}\binom{2j}j^2\binom{3j}j16^{k-j}\notag\\
&=\sum_{j=0}^{p-1}\frac{\binom{2j}j^2\binom{3j}j}{(-16)^j}\sum_{k=j}^{p-1}(3k+1)\binom{k+2j}{3j}.
\end{align*}
By loading the package \texttt{Sigma} in the software \texttt{Mathematica}, we have the following identity:
$$
\sum_{k=j}^{n-1}(3k+1)\binom{k+2j}{3j}=\frac{(3n-1)(3j+1)}{3j+2}\binom{n+2j}{3j+1}.$$
Thus, we have
\begin{equation*}
\sum_{k=0}^{p-1}(3k+1)\frac{D_k}{16^k}=\sum_{j=0}^{p-1}\frac{\binom{2j}j^2\binom{3j}j}{(-16)^j}\frac{(3p-1)(3j+1)}{3j+2}\binom{p+2j}{3j+1}.
\end{equation*}
It is known that $\binom{2j}j\equiv0\pmod p$ for each ${p+1}/2\leq j\leq p-1$, so combining Lemma \ref{p2j}, $p\equiv1\pmod3$, (\ref{zhu13}) and (\ref{zhong1}), we can obtain that
\begin{align*}
&\frac1{3p-1}\sum_{k=0}^{p-1}(3k+1)\frac{D_k}{16^k}\equiv p\sum_{j=0}^{(p-1)/2}\frac{\binom{2j}j^2}{16^j}\frac{1+pH_{2j}-pH_j}{3j+2}+\\
&2p^2\sum_{j=\frac{p+1}2}^{p-1}\frac{\binom{2j}j^2}{16^j}\frac{H_{2j}-H_j}{3j+2}\equiv -\frac52\frac{4p^2}{\binom{\frac{p-1}2}{\frac{p-1}6}^2}+p\binom{-\frac12}{\frac{2p-2}3}^2\left(H_{\frac{4p-4}3}-H_{\frac{2p-2}3}\right)\\
&=-10p^2\binom{\frac{p-1}2}{\frac{p-1}6}^{-2}+\binom{-\frac12}{\frac{2p-2}3}^2\pmod{p^3},
\end{align*}
where we used
$$
\binom{-\frac12}{\frac{2p-2}3}\equiv0\pmod p\ \ \ \mbox{and}\ \ \ p\left(H_{\frac{4p-4}3}-H_{\frac{2p-2}3}\right)\equiv1\pmod p.
$$
It is easy to see that
\begin{align*}
&\binom{-1/2}{(2p-2)/3}^2=\frac{(-\frac12)^2(-\frac12-1)^2\cdots(-\frac12-\frac{2p-2}3+1)^2}{(\frac{2p-2}3)!^2}\\
=&\frac{(\frac12)^2(\frac32)^2\cdots(\frac{p}2-1)^2\frac{p^2}4(\frac{p}2+1)^2\cdots(\frac{p}2+\frac{p-7}6)^2}{(\frac{2p-2}3)!^2}\\
=&\frac{(\frac{p}2-\frac{p-1}2)^2\cdots(\frac{p}2-1)^2\frac{p^2}4(\frac{p}2+1)^2\cdots(\frac{p}2+\frac{p-7}6)^2}{(\frac{2p-2}3)!^2}\\
\equiv&\frac{\frac{p^2}4(\frac{p-1}2)!^2(\frac{p-7}6)!^2}{(\frac{2p-2}3)!^2}=\frac{9p^2}{(p-1)^2}\frac1{\binom{\frac{2p-2}3}{\frac{p-1}2}^2}\equiv\frac{9p^2}{\binom{\frac{2p-2}3}{\frac{p-1}2}^2}\equiv\frac{9p^2}{\binom{\frac{p-1}2}{\frac{p-1}6}^2}\pmod{p^3}.
\end{align*}
Hence
$$
\sum_{k=0}^{p-1}(3k+1)\frac{D_k}{16^k}\equiv p^2\binom{\frac{p-1}2}{\frac{p-1}6}^{-2}\pmod{p^3}.
$$
Therefore, we get the desired result
$$
\sum_{k=0}^{p-1}(3k+2)\frac{D_k}{4^k}\equiv2\sum_{k=0}^{p-1}(3k+1)\frac{D_k}{16^k}\equiv2p^2\binom{\frac{p-1}2}{\frac{p-1}6}^{-2}\pmod{p^3}.
$$
\noindent Now the proof of Theorem \ref{Th1.2} is complete.\qed
\section{Proof of Theorem \ref{Th1.3}}
\setcounter{lemma}{0}
\setcounter{theorem}{0}
\setcounter{corollary}{0}
\setcounter{remark}{0}
\setcounter{equation}{0}
\setcounter{conjecture}{0}
{\it Proof of Theorem \ref{Th1.3}}. Similar as above, by (\ref{2.1}), we have
\begin{align*}
\sum_{k=0}^{p-1}k^2\frac{D_k}{4^k}&=\sum_{k=0}^{p-1}\frac{k^2}{4^k}\sum_{j=0}^{\lfloor k/2\rfloor}\binom{k+j}{3j}\binom{2j}j^2\binom{3j}j4^{k-2j}\notag\\
&=\sum_{j=0}^{(p-1)/2}\frac{\binom{2j}j^2\binom{3j}j}{16^j}\sum_{k=2j}^{p-1}k^2\binom{k+j}{3j}.
\end{align*}
By \texttt{Sigma}, we have the following identity:
\begin{align*}
&\sum_{k=2j}^{n-1}k^2\binom{k+j}{3j}\\
=&\frac{1-j^2-n(2j+3)(3j+1)+n^2(3j+1)(3j+2)}{(3j+2)(3j+3)}\binom{n+j}{3j+1}.
\end{align*}
Thus,
\begin{align*}
&\sum_{k=0}^{p-1}k^2\frac{D_k}{4^k}\\
=&\sum_{j=0}^{\frac{p-1}2}\frac{\binom{2j}j^2\binom{3j}j\binom{p+j}{3j+1}}{16^j}\frac{1-j^2-p(2j+3)(3j+1)+p^2(3j+1)(3j+2)}{(3j+2)(3j+3)}.
\end{align*}
In view of Lemma \ref{Lem2.2}, if $p\equiv1\pmod3$, then we have
\begin{align*}
&\sum_{k=0}^{p-1}k^2\frac{D_k}{4^k}\notag\\
\equiv&\sum_{k=0}^{\frac{p-1}2}\frac{\binom{2k}k^2}{16^k}\frac{p(1-k^2)-p^2(1-k^2)(H_{2k}-H_k)-p^2(2k+3)(3k+1)}{(3k+1)(3k+2)(3k+3)}\\
\equiv&\frac{p}9\sum_{k=0}^{\frac{p-1}2}\frac{\binom{2k}k^2}{16^k}\left(\frac{4}{3k+1}-\frac{5}{3k+2}\right)-\frac{p^2}3\sum_{k=0}^{\frac{p-1}2}\frac{\binom{2k}k^2}{16^k}\left(\frac{5}{3k+2}-\frac{1}{k+1}\right)\\
&-\frac{p^2}9\sum_{k=0}^{\frac{p-1}2}\frac{\binom{2k}k^2(H_{2k}-H_k)}{16^k}\left(\frac{4}{3k+1}-\frac{5}{3k+2}\right)\pmod{p^3}.
\end{align*}
In view of the process of proving Theorems \ref{Th1.1} and \ref{Th1.2}, \cite[(3.5)]{s2018} and \cite{YMK}, We have
 \begin{align*}
&\sum_{k=0}^{p-1}k^2\frac{D_k}{4^k}\notag\\
\equiv&\frac{p}9\sum_{k=0}^{\frac{p-1}2}\frac{\binom{2k}k^2}{16^k}\left(\frac{4}{3k+1}-\frac{5}{3k+2}\right)+\frac{5p^2}9\sum_{k=0}^{\frac{p-1}2}\frac{\binom{2k}k^2(H_{2k}-H_k)}{(3k+2)16^k}\\
\equiv&\frac{4}9\left(4x^2-2p-\frac{p^2}{4x^2}\right)-\frac{5}9\frac{-4p^2}{4x^2}+\frac59\frac32\frac{-4p^2}{4x^2}\\
=&\frac{16x^2}9-\frac{8p}9-\frac{7p^2}{18x^2}\pmod{p^3}.
\end{align*}
If $p\equiv2\pmod3$, then modulo $p^2$, we have
\begin{align*}
&\sum_{k=0}^{p-1}k^2\frac{D_k}{4^k}\notag\\
\equiv&\sum_{k=0}^{\frac{p-1}2}\frac{\binom{2k}k^2}{16^k}\frac{p(1-k^2)-p^2(1-k^2)(H_{2k}-H_k)-p^2(2k+3)(3k+1)}{(3k+1)(3k+2)(3k+3)}\\
\equiv&\frac{p}9\sum_{k=0}^{\frac{p-1}2}\frac{\binom{2k}k^2}{16^k}\left(\frac{4}{3k+1}-\frac{5}{3k+2}\right)-\frac{p^2}3\sum_{k=0}^{\frac{p-1}2}\frac{\binom{2k}k^2}{16^k}\left(\frac{5}{3k+2}-\frac{1}{k+1}\right)\\
&-\frac{p^2}9\sum_{k=0}^{\frac{p-1}2}\frac{\binom{2k}k^2(H_{2k}-H_k)}{16^k}\left(\frac{4}{3k+1}-\frac{5}{3k+2}\right)\\
\equiv&-\frac{5p(1+3p)}9\sum_{k=0}^{\frac{p-1}2}\frac{\binom{2k}k^2}{16^k(3k+2)}+\frac{5p^2}9\sum_{k=0}^{\frac{p-1}2}\frac{\binom{2k}k^2(H_{2k}-H_k)}{16^k(3k+2)}.
\end{align*}
And similar as above, we have
\begin{align}\label{p23p2}
\frac{2p}{3p+1}\frac{(\frac13)_{\frac{p-1}2}}{(\frac23)_{\frac{p-1}2}}\equiv\binom{\frac{2p-1}3}{\frac{p-1}2}\binom{\frac{p-1}2}{\frac{p+1}6}(-1)^{\frac{p+1}6}(1+2p-\frac{2p}3q_p(2))\pmod{p^2}.
\end{align}
and
$$
\sum_{k=1}^{\frac{p-1}2}\frac{(\frac23)_k}{k(\frac13)_k}\equiv-3\sum_{k=0}^{\frac{p-5}3}\frac1{3k+1}\equiv-3\sum_{k=1\atop k\equiv1\pmod3}^{p-1}\frac1k-3\equiv-3\pmod{p}.
$$
So
\begin{align}\label{h2khk16k}
\sum_{k=0}^{\frac{p-1}2}\frac{\binom{2k}k^2(H_{2k}-H_k)}{16^k(3k+2)}\equiv3\sum_{k=0}^{\frac{p-1}2}\frac{\binom{2k}k^2}{16^k(3k+2)}\pmod p.
\end{align}
Hence
\begin{align*}
&\sum_{k=0}^{p-1}k^2\frac{D_k}{4^k}\equiv-\frac{5p}9\sum_{k=0}^{\frac{p-1}2}\frac{\binom{2k}k^2}{16^k(3k+2)}\\
\equiv&-\frac{5}9\binom{\frac{2p-1}3}{\frac{p-1}2}\binom{\frac{p-1}2}{\frac{p+1}6}(-1)^{\frac{p+1}6}(1+2p-\frac{2p}3q_p(2))\pmod{p^2}.
\end{align*}
It is easy to check that
\begin{align}\label{zuhep2}
&\binom{\frac{2p-1}3}{\frac{p-1}2}\binom{\frac{p-1}2}{\frac{p+1}6}(-1)^{\frac{p+1}6}\equiv\binom{\frac{p-1}2}{\frac{p+1}6}^2\left(1+2pq_p(2)-\frac{3p}2q_p(3)\right)\notag\\
&\equiv4\binom{\frac{p-1}2}{\frac{p-5}6}^2\left(1+2pq_p(2)-\frac{3p}2q_p(3)\right)\pmod{p^2}.
\end{align}
Thus,
\begin{align*}
\sum_{k=0}^{p-1}k^2\frac{D_k}{4^k}&\equiv-\frac{20}9\binom{\frac{p-1}2}{\frac{p-5}6}^2(1+2p+\frac{4p}3q_p(2)-\frac{3p}2q_p(3))\\
&=-\frac{20}9R_3(p)\pmod{p^2}.
\end{align*}
So we obtain the first congruence in Theorem \ref{Th1.3}.

\noindent Now we consider the second congruence in Theorem \ref{Th1.3}. Similar as above, by (\ref{2.0}), we have
\begin{align*}
\sum_{k=0}^{p-1}k^2\frac{D_k}{16^k}&=\sum_{k=0}^{p-1}\frac{k^2}{16^k}\sum_{j=0}^{k}(-1)^j\binom{k+2j}{3j}\binom{2j}j^2\binom{3j}j16^{k-j}\notag\\
&=\sum_{j=0}^{p-1}\frac{\binom{2j}j^2\binom{3j}j}{(-16)^j}\sum_{k=j}^{p-1}k^2\binom{k+2j}{3j}.
\end{align*}
By \texttt{Sigma}, we have the following identity:
\begin{align*}
&\sum_{k=j}^{n-1}k^2\binom{k+2j}{3j}\\
=&\frac{1+3j+2j^2-n(4j+3)(3j+1)+n^2(3j+1)(3j+2)}{(3j+2)(3j+3)}\binom{n+2j}{3j+1}.
\end{align*}
Thus, if $p\equiv1\pmod3$, then modulo $p^3$, we have
\begin{align*}
&\sum_{k=0}^{p-1}k^2\frac{D_k}{16^k}+\frac1{18p(2p+1)}\binom{-1/2}{\frac{2p-2}3}^2\binom{2p-2}{\frac{2p-2}3}\binom{p+\frac{4p-4}3}{2p-1}\\
\equiv&\sum_{j=0}^{\frac{p-1}2}\frac{\binom{2j}j^2\binom{3j}j\binom{p+2j}{3j+1}}{(-16)^j}\frac{1+3j+2j^2-p(4j+3)(3j+1)+p^2(3j+1)(3j+2)}{(3j+2)(3j+3)}.
\end{align*}
Hence, similar as above, we have
\begin{align*}
&\sum_{k=0}^{p-1}k^2\frac{D_k}{16^k}+\frac1{18p(2p+1)}\binom{-1/2}{\frac{2p-2}3}^2\binom{2p-2}{\frac{2p-2}3}\binom{p+\frac{4p-4}3}{2p-1}\\
\equiv&\frac{p}9\sum_{j=0}^{\frac{p-1}2}\frac{\binom{2j}j^2}{16^j}\left(\frac{1}{3j+1}+\frac1{3j+2}\right)-\frac{p^2}3\sum_{j=0}^{\frac{p-1}2}\frac{\binom{2j}j^2}{16^j}\left(\frac{1}{3j+2}+\frac1{j+1}\right)\\
&+\frac{p^2}9\sum_{j=0}^{\frac{p-1}2}\frac{\binom{2j}j^2}{16^j}\left(\frac{1}{3j+1}+\frac1{3j+2}\right)(H_{2j}-H_j)\pmod{p^3}.
\end{align*}
In view of the process of proving Theorems \ref{Th1.1} and \ref{Th1.2}, \cite[(3.5)]{s2018} and \cite{YMK}, We have
\begin{align*}
&\sum_{k=0}^{p-1}k^2\frac{D_k}{16^k}+\frac1{18p(2p+1)}\binom{-1/2}{\frac{2p-2}3}^2\binom{2p-2}{\frac{2p-2}3}\binom{p+\frac{4p-4}3}{2p-1}\\
\equiv&\frac19\left((4x^2-2p-\frac{p^2}{4x^2}\right)+\frac19\frac32\frac{-4p^2}{4x^2}+\frac19\frac{-4p^2}{4x^2}\\
\equiv&\frac{4x^2}9-\frac{2p}9-\frac{11p^2}{36x^2}\pmod{p^3}.
\end{align*}
It is easy to see that
$$
\binom{2p-2}{\frac{2p-2}3}\binom{p+\frac{4p-4}3}{2p-1}\equiv-2p\pmod{p^2}.
$$
This, with the above $\binom{-1/2}{\frac{2p-2}3}^2\equiv9p^2/(4x^2)\pmod{p^3}$, we immediately get that
$$
\sum_{k=0}^{p-1}k^2\frac{D_k}{16^k}\equiv\frac{4x^2}9-\frac{2p}9-\frac{11p^2}{36x^2}-\left(-\frac{p^2}{4x^2}\right)=\frac{4x^2}9-\frac{2p}9-\frac{p^2}{18x^2}\pmod{p^3}.
$$
If $p\equiv2\pmod3$, then modulo $p^2$, we have
\begin{align*}
&\sum_{k=0}^{p-1}k^2\frac{D_k}{16^k}\\
\equiv&\sum_{j=0}^{\frac{p-1}2}\frac{\binom{2j}j^2\binom{3j}j\binom{p+2j}{3j+1}}{(-16)^j}\frac{1+3j+2j^2-p(4j+3)(3j+1)+p^2(3j+1)(3j+2)}{(3j+2)(3j+3)}.
\end{align*}
Hence, similar as above, we have
\begin{align*}
&\sum_{k=0}^{p-1}k^2\frac{D_k}{16^k}\\
\equiv&\frac{p}9\sum_{j=0}^{\frac{p-1}2}\frac{\binom{2j}j^2}{16^j}\left(\frac{1}{3j+1}+\frac1{3j+2}\right)-\frac{p^2}3\sum_{j=0}^{\frac{p-1}2}\frac{\binom{2j}j^2}{16^j}\left(\frac{1}{3j+2}+\frac1{j+1}\right)\\
&+\frac{p^2}9\sum_{j=0}^{\frac{p-1}2}\frac{\binom{2j}j^2}{16^j}\left(\frac{1}{3j+1}+\frac1{3j+2}\right)(H_{2j}-H_j)\\
\equiv&\frac{p}9\sum_{j=0}^{\frac{p-1}2}\frac{\binom{2j}j^2}{(3j+2)16^j}-\frac{p^2}3\sum_{j=0}^{\frac{p-1}2}\frac{\binom{2j}j^2}{(3j+2)16^j}+\frac{p^2}9\sum_{j=0}^{\frac{p-1}2}\frac{\binom{2j}j^2}{(3j+2)16^j}(H_{2j}-H_j)\\
\equiv&\frac{p}9\sum_{j=0}^{\frac{p-1}2}\frac{\binom{2j}j^2}{(3j+2)16^j}=-\frac15\sum_{k=0}^{p-1}k^2\frac{D_k}{4^k}=\frac49R_3(p)\pmod{p^2}.
\end{align*}
\noindent Now the proof of the second congruence in Theorem \ref{Th1.3} is complete.

\noindent{\it Proof of (\ref{kdk4k})}. Similarly, 
\begin{align*}
\sum_{k=0}^{p-1}k\frac{D_k}{4^k}&=\sum_{k=0}^{p-1}\frac{k}{4^k}\sum_{j=0}^{\lfloor k/2\rfloor}\binom{k+j}{3j}\binom{2j}j^2\binom{3j}j4^{k-2j}\\
&=\sum_{j=0}^{\frac{p-1}2}\frac{\binom{2j}j^2\binom{3j}j}{16^j}\sum_{k=2j}^{p-1}k\binom{k+j}{3j}.
\end{align*}
By \texttt{Sigma}, we find the following identity which can be proved by induction on $n$:
$$
\sum_{k=2j}^{n-1}k\binom{k+j}{3j}=\frac{3nj+n-j-1}{3j+2}\binom{n+j}{3j+1}.
$$
Hence
\begin{align*}
\sum_{k=0}^{p-1}k\frac{D_k}{4^k}=\sum_{j=0}^{\frac{p-1}2}\frac{\binom{2j}j^2\binom{3j}j}{16^j}\frac{3pj+p-j-1}{3j+2}\binom{p+j}{3j+1}.
\end{align*}
In view of Lemma \ref{Lem2.2}, and $p\equiv2\pmod3$, then modulo $p^2$ we have
\begin{align*}
&\sum_{k=0}^{p-1}k\frac{D_k}{4^k}\equiv\sum_{k=0}^{\frac{p-1}2}\frac{\binom{2k}k^2}{16^k}\frac{p(-1-k)+p^2(1+k)(H_{2k}-H_k)+p^2(3k+1)}{(3k+1)(3k+2)}\\
\equiv&-\frac{p}3\sum_{k=0}^{\frac{p-1}2}\frac{\binom{2k}k^2}{16^k}\left(\frac{2}{3k+1}-\frac{1}{3k+2}\right)+p^2\sum_{k=0}^{\frac{p-1}2}\frac{\binom{2k}k^2}{(3k+2)16^k}\\
&\ \ \ \ \ \ \ +\frac{p^2}3\sum_{k=0}^{\frac{p-1}2}\frac{\binom{2k}k^2(H_{2k}-H_k)}{16^k}\left(\frac{2}{3k+1}-\frac{1}{3k+2}\right)\\
\equiv&\frac{p}3\sum_{k=0}^{\frac{p-1}2}\frac{\binom{2k}k^2}{(3k+2)16^k}+p^2\sum_{k=0}^{\frac{p-1}2}\frac{\binom{2k}k^2}{(3k+2)16^k}-\frac{p^2}3\sum_{k=0}^{\frac{p-1}2}\frac{\binom{2k}k^2(H_{2k}-H_k)}{(3k+2)16^k}.
\end{align*}
By (\ref{p23p2})-(\ref{zuhep2}), we have
\begin{align*}
\sum_{k=0}^{p-1}k\frac{D_k}{4^k}\equiv&\frac{p}3\sum_{k=0}^{\frac{p-1}2}\frac{\binom{2k}k^2}{(3k+2)16^k}\\
\equiv&\frac13\binom{\frac{2p-1}3}{\frac{p-1}2}\binom{\frac{p-1}2}{\frac{p+1}6}(-1)^{\frac{p+1}6}\left(1+2p-\frac{2p}3q_p(2)\right)\\
\equiv&\frac43R_3(p)\pmod{p^2}.
\end{align*}
Similarly, 
\begin{align*}
\sum_{k=0}^{p-1}k\frac{D_k}{16^k}&=\sum_{k=0}^{p-1}\frac{k}{16^k}\sum_{j=0}^{k}(-1)^j\binom{k+2j}{3j}\binom{2j}j^2\binom{3j}j16^{k-j}\\
&=\sum_{j=0}^{p-1}\frac{\binom{2j}j^2\binom{3j}j}{(-16)^j}\sum_{k=j}^{p-1}k\binom{k+2j}{3j}.
\end{align*}
By \texttt{Sigma}, we find the following identity which can be proved by induction on $n$:
$$
\sum_{k=j}^{n-1}k\binom{k+2j}{3j}=\frac{3nj+n-2j-1}{3j+2}\binom{n+2j}{3j+1}.
$$
Hence
\begin{align*}
\sum_{k=0}^{p-1}k\frac{D_k}{16^k}=\sum_{j=0}^{p-1}\frac{\binom{2j}j^2\binom{3j}j}{(-16)^j}\frac{3pj+p-2j-1}{3j+2}\binom{p+2j}{3j+1}.
\end{align*}
It is known that $\binom{2k}k\equiv0\pmod p$ for each $(p+1)/2\leq k\leq p-1$, then by Lemma \ref{p2j} and $p\equiv2\pmod3$, we have the following modulo $p^2$,
\begin{align*}
&\sum_{k=0}^{p-1}k\frac{D_k}{16^k}\equiv\sum_{k=0}^{\frac{p-1}2}\frac{\binom{2k}k^2}{16^k}\frac{p(-1-2k)-p^2(1+2k)(H_{2k}-H_k)+p^2(3k+1)}{(3k+1)(3k+2)}\\
&\equiv-\frac{p}3\sum_{k=0}^{\frac{p-1}2}\frac{\binom{2k}k^2}{16^k}\left(\frac{1}{3k+1}+\frac{1}{3k+2}\right)+p^2\sum_{k=0}^{\frac{p-1}2}\frac{\binom{2k}k^2}{(3k+2)16^k}\\
&\ \ \ \ \ \ \ -\frac{p^2}3\sum_{k=0}^{\frac{p-1}2}\frac{\binom{2k}k^2(H_{2k}-H_k)}{16^k}\left(\frac{1}{3k+1}+\frac{1}{3k+2}\right)\\
&\equiv-\frac{p}3\sum_{k=0}^{\frac{p-1}2}\frac{\binom{2k}k^2}{(3k+2)16^k}+p^2\sum_{k=0}^{\frac{p-1}2}\frac{\binom{2k}k^2}{(3k+2)16^k}-\frac{p^2}3\sum_{k=0}^{\frac{p-1}2}\frac{\binom{2k}k^2(H_{2k}-H_k)}{(3k+2)16^k}.
\end{align*}
By (\ref{h2khk16k}), we have
\begin{align*}
\sum_{k=0}^{p-1}k\frac{D_k}{16^k}\equiv-\frac{p}3\sum_{k=0}^{\frac{p-1}2}\frac{\binom{2k}k^2}{(3k+2)16^k}\equiv-\sum_{k=0}^{p-1}k\frac{D_k}{4^k}\pmod{p^2}.
\end{align*}
Therefore,
$$
\sum_{k=0}^{p-1}k\frac{D_k}{4^k}\equiv-\sum_{k=0}^{p-1}k\frac{D_k}{16^k}\equiv\frac43R_3(p)\pmod{p^2}.
$$
Now we finish the proof of Theorem \ref{Th1.3}.\qed

\Ack. The first author is funded by the National Natural Science Foundation of China (12001288) and China Scholarship Council (202008320187).


\begin{thebibliography}{ST10}
\bibitem{CCL} H.H. Chan, S.H. Chan and Z.-G. Liu, {\it Domb's numbers and Ramanujan-Sato type series for $1/\pi$}, Adv. in Math. {\bf 186} (2004), 396--410.
\bibitem{CZ} H.H. Chan and W. Zudilin, {\it New representations for Ap\'{e}ry-like sequences}, Mathematika {\bf 56} (2010), 107--117.
\bibitem{L} E. Lehmer, {\it On congruences involving Bernoulli numbers and the quotients of Fermat and Wilson}, Ann. Math.  {\bf 39}(1938), 350--360.
\bibitem{liud} J.-C. Liu, {\it On Supercongruences for sums involving Domb numbers}, preprint, arXiv:2008.02647v2.
\bibitem{mao} G.-S. Mao, {\it On some congruences of binomial coefficients modulo $p^3$ with applications}, preprint, temporarily on Researchgate, Doi:10.13140/RG.2.2.12033.17766.
\bibitem{mw}G.-S. Mao and J. Wang, {\it On some congruences involving Domb numbers and harmonic numbers}, Int. J. Number Theory, {\bf 15} (2019), 2179--2200.
\bibitem{ml} G.-S. Mao and Y. Liu, {On two congruence conjectures of Z.-W. Sun involving Franel numbers}, preprint, arXiv:2111.08775.
\bibitem{mp3} G.-S. Mao and W. R. Zhu, {\it On some conjectural congruences involving Ap\'{e}ry-like numbers $V_n$}, preprint, temporarily on Researchgate, Doi:10.13140/RG.2.2.12308.42880.
\bibitem{MS18} Y.-P. Mu and Z.-W. Sun, {\it Telescoping method and congruences for double sums}, Int. J. Number Theory {\bf 14} (2018), no.1, 143--165.
\bibitem{Murty02} M. R. Murty, {\it Introduction to $p$-adic analytic number theory}, AMS/IP Studies in Advanced Mathematics, 27, American Mathematical Society, Providence, RI; International Press, Somerville, MA, 2002.
\bibitem{Robert00} A. M. Robert, {\it  A course in $p$-adic analysis}, Graduate Texts in Mathematics, 198. Springer-Verlag, New York, 2000.
\bibitem{R}  M. D. Rogers, {\it New ${}_5F_4$ hypergeometric transformations, three-variable Mahler measures, and formulas for $1/\pi$}, Ramanujan J. {\bf 18} (2009), 327--340.
\bibitem{S}C. Schneider, {\it Symbolic summation assists combinatorics}, S\'em. Lothar. Combin. {\bf 56} (2007), Article B56b.
\bibitem{s2000} Z.-H. Sun, {\it Congruences concerning Bernoulli numbers and Bernoulli polynomials}, Discrete Appl. Math. {\bf 105} (2000), no.1-3, 193--223.
\bibitem{s2008} Z.-H. Sun, {\it Congruences involving Bernoulli and Euler numbers}, J. Number Theory {\bf 128} (2008), no.2, 280--312.
\bibitem{sund} Z.-H. Sun, {\it Congruences for Domb and Almkvist-Zudilin numbers}, Integral Transforms Spec. Funct. {\bf 26} (2015), no.8, 642--659.
\bibitem{sijnt} Z.-H. Sun, {\it Super congruences concerning Bernoulli polynomials}, Int. Differ. Equa. Appl. {\bf 24} (2018), no.10, 1685--1713.
\bibitem{s2018} Z.-H. Sun, {\it Super congruences for two Ap\'{e}ry-like sequences}, J. Number Theory {\bf 11} (2015), no.8, 2393--2404.
\bibitem{SH2} Z.-H. Sun, {\it Congruences involving binomial coefficients and Ap\'{e}ry-like numbers}, Publ. Math. Debrecen {\bf 96} (2020), no.3--4, 315--346.
\bibitem{S2020} Z.-H. Sun, {\it Supercongruences and binary quadratic forms}, Acta Arith. {\bf 199} (2021). no.1, 1--32.
\bibitem{s2111} Z.-H. Sun, {\it New conjectures invloving binomial coefficients and Apery-like numbers}, preprint, arXiv:2111.04538v1.
\bibitem{S14}Z.-W. Sun, {\it Number Theory and Related Area} (eds., Y. Ouyang, C. Xing, F. Xu and P. Zhang), Adv. Lecr. Math. {\bf 27}, Higher Education Press and International Press, Beijing-Boston, 2013, pp. 149--197.
\bibitem{swz} B. Sury, T.-M. Wang and F.-Z. Zhao, {\it Identities involving reciprocals of binomial coefficients}, J. Integer Seq. {\bf 7} (2004), Article 04.2.8.
\bibitem{YMK} K. M. Yeung, {\it On congruences for Binomial Coefficients}, J. Number Theory {\bf 33} (1989), 1--17.
\end{thebibliography}
\end{document}